\documentclass{amsart}

\usepackage{amsfonts,amssymb,mathtools}	
\usepackage{url}
\usepackage{hyperref}
\usepackage[usenames,dvipsnames]{xcolor}

\usepackage{hyperref}
\hypersetup{
    unicode=false,          		
    pdftoolbar=true,        		
    pdfmenubar=true,        	
    pdffitwindow=false,     		
    pdfstartview={FitH},    		
    linktocpage=true,			
    pdfnewwindow=true,      	
    colorlinks=true,       			
    linkcolor=blue,          		
    citecolor=PineGreen,    	
    filecolor=magenta,      		
    urlcolor=cyan           		
}

\theoremstyle{plain}
\newtheorem{theorem}{Theorem}[section]

\newtheorem{lemma}[theorem]{Lemma}
\newtheorem{proposition}[theorem]{Proposition}

\theoremstyle{definition}

\theoremstyle{remark}

\numberwithin{equation}{section}

\DeclareMathOperator{\essinf}{ess \,\, inf}
\DeclareMathOperator{\esssup}{ess \,\, sup}

\begin{document}
\title[Calder\'{o}n's Problem with Discontinuous Conductivities]
{Nachman's Reconstruction Method for the Calder\'on problem with discontinuous conductivities}
\author{George Lytle}
\author{Peter Perry}
\author{Samuli Siltanen}
\date{\today}
\begin{abstract}
We show that Nachman's integral equations for the Calder\'{o}n problem, derived for conductivities in $W^{2,p}(\Omega)$, still hold for $L^\infty$ conductivities which are $1$ in a neighborhood of the boundary. We also prove convergence of scattering transforms for smooth approximations to the scattering transform of $L^\infty$ conductivities. We rely on Astala-P\"{a}iv\"{a}rinta's formulation of the Calder\'{o}n problem for a framework in which these convergence results make sense.
\end{abstract}
\maketitle
\tableofcontents

%
%


%
%


\newcommand{\dee}{\partial}
\newcommand{\dbar}{\overline{\partial}}
\newcommand{\dotarg}{\, \cdot \, }
\newcommand{\eps}{\varepsilon}
\newcommand{\lam}{\lambda}
\newcommand{\Lam}{\Lambda}

\newcommand{\sig}{\sigma}

\newcommand{\dint}{\displaystyle{\int}}

\newcommand{\vbar}{\overline{v}}


\newcommand{\kbar}{\overline{k}}
\newcommand{\xbar}{\overline{x}}
\newcommand{\zbar}{\overline{z}}


\newcommand{\bft}{\mathbf{t}}

\newcommand{\C}{\mathbb{C}}
\newcommand{\D}{\mathbb{D}}
\newcommand{\N}{\mathbb{N}}
\newcommand{\R}{\mathbb{R}}

\newcommand{\calH}{\mathcal{H}}
\newcommand{\calL}{\mathcal{L}}
\newcommand{\calP}{\mathcal{P}}

\newcommand{\bD}{\partial \mathbb{D}}


\newcommand{\norm}[2][ ]{\left\| {#2} \right\|_{ #1 }}
\newcommand{\bigO}[1]{\mathcal{O}\left( {#1} \right)}
\newcommand{\dashint}[1]{\int_{#1} {\hspace{-16.55pt} - \hspace{10pt}}}   


\newcommand{\loc}{\mathrm{loc}}
\newcommand{\bdrytrace}[1]{\left. {#1} \right|_{\partial \mathbb{D}}}
\newcommand{\KariLassi}{{Astala-\Lassi}}
\newcommand{\AP}{Astala-P\"{a}iv\"{a}rinta }

\section{Introduction}

Calder\'{o}n's inverse conductivity problem \cite{Calderon80} is to reconstruct  the conductivity $\sigma$ of a conducting body $\Omega$ from boundary measurements. The electrical potential $u$ obeys
the equation
\begin{equation}
\label{bvp}
\begin{aligned}
\nabla \cdot \left( \sigma \nabla u \right)	&=	0\\
\left. u \right|_{\dee \Omega}	&=	 f
\end{aligned}
\end{equation}
where $f \in H^{1/2}(\dee \Omega)$ is the potential on the boundary. The induced current through the boundary
is given by the Dirichlet to Neumann map
\begin{equation}
\label{D.to.N}
\Lam_\sigma f = \sigma \left. \frac{\dee u}{ \dee \nu} \right|_{\dee \Omega}
\end{equation}
where $\nu$ is the outward normal. The Dirichlet-to-Neumann map is known to be bounded from $H^{1/2}(\dee \Omega)$ to $H^{-1/2}(\dee \Omega)$. 

The identifiability problem for Calder\'{o}n's inverse problem has been extensively studied, culminating in Astala-P\"{a}iv\"{a}rinta's proof \cite{AP06a}  that the operator $\Lam_\sig$ determines $\sig$ uniquely provided only that $\sig \in L^\infty$ has essential infimum bounded below by a strictly positive constant. On the other hand, computational methods to recover the conductivity are based on Nachman's \cite{Nachman96} integral equations derived in his analysis of Calder\'{o}n's problem via two-dimensional inverse scattering theory for Schr\"{o}dinger's equation at zero energy. The Schr\"{o}dinger potential corresponding to conductivity $\sigma$ is 
$q= (\Delta \sqrt{\sig})/\sqrt{\sig}$, so that Nachman must assume $\sig \in W^{2,p}(\Omega)$ for some $p>1$ in addition to the essential lower bound. 

Numerical algorithms based on Nachman's reconstruction algorithm routinely incorporate a high-frequency cutoff on the scattering transform to make efficient and tractable code \cite{SMI00}.  These codes can be used to implement efficient reconstructions for applications in medical imaging \cite{HVMV15}.   Suprisingly, numerical algorithms based on Nachman's integral equations remain effective for \emph{discontinuous} conductivities provided that such a suitable high-frequency cutoff is imposed on the scattering transform, as shown in recent numerical studies \cite{APRS14}. 

For smooth conductivities, the high-frequency cutoff can be understood as a regularization technique for the inversion process \cite{KLMS09}. This paper is the first in a series designed to explain, from the analytical point of view, why this is  also be the case for non-smooth conductivities.

Nachman's method uses the given Dirichlet-to-Neumann map to compute the \emph{scattering transform} $\bft$ of the Schr\"{o}dinger potential $q$,  regarded as a potential on $\R^2$ by extending to $0$ outside $\Omega$, and recover $\sig$ from the scattering transform. The potential $q$  has $\sqrt{\sig}$ as its ground state solution if we similarly extend $\sig(x)$ by setting $\sig(x)=1$ on $\R^2 \setminus \Omega$. Known $\dbar$-methods in two-dimensional Schr\"{o}dinger scattering allow one to recover the scattering eigenfunctions (and hence, the ground state) from the scattering data by solving a $\dbar$ problem. 

More precisely, given the potential $q$, one computes the scattering transform $\bft$ from Faddeev's \cite{Faddeev65} complex geometric optics (CGO) solutions of the the Schr\"{o}dinger equation:
\begin{equation}
\label{CGO}
\begin{aligned}
(-\Delta + q )\psi	&=	0,	\\
\lim_{|x| \to \infty} \psi(x,k) e^{-ikx} - 1 &=	0,
\end{aligned}
\end{equation}
where $k=k_1+ik_2$ is a complex parameter and $kx$ denotes complex multiplication of $k$ with $x=x_1+ix_2$. These exponentially growing solutions determine the scattering transform $\bft$ through the integral formula
\begin{equation}
\label{t}
\bft(k) = \int_{\R^2} e_k(x) q(x) m(x,k) \, dx
\end{equation}
where $m(x,k) = e^{-ikx} \psi(x,k)$ and $e_k(x) = \exp \left(i\left(kx + \kbar \xbar\right)\right)$. 
Given $\bft$ one can recover $m(x,k)$ (and hence $m(x,0) = \sqrt{\sigma(x)}$) by solving the $\dbar$ problem 
\begin{equation}
\label{N.dbar}
\begin{aligned}
\dbar_k m(x,k)	&=	\frac{\bft(k)}{4\pi \kbar} \, e_{-k}(x) \overline{m(x,k)}\\
\lim_{|k| \to \infty}	\left( m(x,k) - 1 \right)	&-	0
\end{aligned}
\end{equation}
where $\dbar_k = (1/2)\left( \dee_{k_1} + i \dee_{k_2}\right)$. 

The key point is that the scattering transform of $q$ can be determined directly from the Dirichlet-to-Neumann operator. Because $q$ has compact support, one can reduce \eqref{CGO} and \eqref{t} to the boundary integral equations
\begin{align}
\label{N.bdryint}	
\left. \psi \right|_{\dee \Omega}	
	&=	 \left.e^{ikx}\right|_{\dee \Omega}	- S_k \left(\Lam_q - \Lam_0 \right) 
								\left( 
										\left. \psi \right|_{\dee \Omega} 
								\right). \\
\label{N.bdry.t}
\bft(k)	&=	\int_{\dee \Omega} e^{i \kbar \xbar} 
								\left( \Lam_q - \Lam_0 \right) 
								\left( 
										\left. \psi \right|_{\dee \Omega} 
								\right) \, ds
\end{align}
Here $S_k$ is convolution with the Fadeev Green's function on $\dee \Omega$, an integral operator described in Section \ref{sec:prelim}.  The operator $\Lam_q$ is the Dirichlet to Neumann operator for the \emph{Schr\"{o}dinger} problem
\begin{equation}
\label{svp}
\begin{aligned}
(-\Delta + q) \psi	&=	0,	\\
\left. \psi \right|_{\dee \Omega} &= f.
\end{aligned}
\end{equation}
The boundary integral equations \eqref{N.bdryint} were first introduced by R.\ Novikov \cite{Novikov88}.
The operator $\Lambda_0$ is the Dirichlet-to-Neumann operator for harmonic functions on $\Omega$, corresponding to $q(x) \equiv 0$ and $\sig(x) \equiv 1$. Given $\bft$, one can then solve the $\dbar$-problem \eqref{N.dbar} and recover $\sigma$ from 
$$ \sigma(x) = m(x,0)^2. $$

In this paper we show that the integral equations \eqref{N.bdryint} are still uniquely solvable for the Dirichlet-to-Neumann operator of a positive, essentially bounded conductivity with strictly positive essential lower bound. Moreover, we identify the resulting scattering transform as a natural analogue of Nachman's scattering transform which is, in fact, a limit of scattering transforms obtained through monotone approximation by smooth functions. A key ingredient in our analysis is the Beltrami equation of Astala-P\"{a}iv\"{a}rinta and the associated scattering transform, which provides a way of identifying the `scattering transform' that arises from the limit of Nachman's equations.

To describe our results, we first recall a standard reduction due to Nachman \cite[Section 6]{Nachman96}.  Without loss, we may  assume that $\Omega$ is the unit disc $\D$ and that $\sigma(x) \equiv 1$ in a neighborhood of $\D$. We make the second assumption more precise:
\begin{equation}
\label{sigma.one}
\text{There is an }r_1 \in (0,1) \text{ so that }\sigma(x)=1\text{ for }|x| \geq r_1. 
\end{equation}

Next we describe the Astala-P\"{a}iv\"{a}rinta scattering transform which provides the context in which our convergence result can be understood. Given a positive conductivity $\sigma$ with $\sigma(x) \geq c> 0$ a.e., the Beltrami coefficient associated to $\sigma$ is $\mu =  (1-\sig)/(1+\sig)$ and satisfies $|\mu(x)| \leq \kappa < 1$. Moreover, $\mu$ has compact support since $\sig(x)=1$ outside a compact set. For any real solution $u \in H^1(\D)$ of \eqref{bvp}, there exists a real-valued function $v \in H^1(\D)$, called the $\sigma$-harmonic conjugate of $u$, so that $f=u+iv$ solves the Beltrami equation
\begin{equation}
\label{Beltrami}
\dbar f = \mu \dee f
\end{equation}
where 
$\dee =  (1/2)\left( \dee_{x_1} - i \dee_{x_2} \right)$ and 
$\dbar = (1/2)\left( \dee_{x_1} + i \dee_{x_2} \right)$ 
are the operators of differentiation with respect to $x$ and $\xbar$.
Astala and P\"{a}iv\"{a}rinta show that the Beltrami equation \eqref{Beltrami} admits CGO solutions which define a scattering transform analogous to $\bft$ which remains well-defined under the weaker assumption that $\mu \in L^\infty(\Omega)$.

\begin{theorem}  \cite[Theorem 4.2]{AP06a}
Let $\mu \in L^\infty(\D)$ with $ \norm[\infty]{\mu} \leq \kappa < 1$. For each $k \in \C$ and each
$p \in (2,1+\kappa^{-1})$, there exists a unique solution $f_\mu \in W^{1,p}_{\loc}(\R^2)$ 
of \eqref{Beltrami} of the form $f_\mu = e^{ikx} M_\mu(x,k)$ where $M_\mu(x,k) -1 \in W^{1,p}(\R^2)$. 
\end{theorem}

We refer to $M_\mu$ as the normalized CGO solution of \eqref{Beltrami} and denote by $M_{\pm \mu}$ the normalized solutions corresponding to $\mu$ and $-\mu$. The associated scattering transform $\tau_\mu$ is given by
\begin{equation}
\label{tau}
\overline{\tau_\mu(k)} = \frac{1}{2\pi} \int \dbar_x \left(M_\mu(x,k) - M_{-\mu}(x,k)\right) \, dx
\end{equation}
If conductivities are smooth, one has \cite{APRS14}
\begin{equation}
\label{tau-to-t}
\bft(k) = -4\pi i \overline{k}\tau_\mu(k).
\end{equation}

Our first result concerns solvability of the Nachman integral equations for a non-smooth conductivity. Observe that, under our assumption \eqref{sigma.one}, a solution $\psi$ of 
\eqref{svp} generates a solution of the boundary value problem \eqref{bvp} via $u(x) = \sigma(x)^{-1/2} \psi(x)$, and the Dirichlet-to-Neumann operators for \eqref{svp} and \eqref{bvp} are in fact \emph{identical}. Thus, under the assumption \eqref{sigma.one}, we can recast \eqref{N.bdryint} and \eqref{N.bdry.t} in terms of the Dirichlet-to-Neumann operators for the original conductivity problem, taking $\Omega$ to be the unit disc $\D$.
\begin{align}
\label{N.bdryint.bis}
\left. \psi \right|_{\bD}	
	&=	\left.e^{ikx}\right|_{\bD} - S_k\left(\Lam_\sig - \Lam_1\right)\left( \left. \psi \right|_{\bD} \right),\\
\label{N.bdry.t.bis}
\bft(k)	
	&=	\int_{\bD} e^{i\kbar \xbar} \left(\Lam_\sig - \Lam_1\right) \left( \left. \psi \right|_{\bD} \right) \, ds
\end{align}
where by abuse of notation we write $\Lam_1$, the Dirichlet-to-Neumann operator corresponding to  \eqref{bvp} with $\sig(x)=1$, instead of $\Lam_0$, the Dirichlet-to-Neumann operator corresponding to \eqref{svp} with $q=0$, which are the same thing. Our first result is that \eqref{N.bdryint.bis} is uniquely solvable for $\sigma \in L^\infty$ with strictly positive essential infimum and any $k$.

\begin{theorem}
\label{thm:N.solve} Let $\sig \in L^\infty(\D)$ with $\sig(x) \geq c$ for a fixed $c>0$, and suppose that \eqref{sigma.one} holds. For each $k \in \C$, there exists a unique $g \in H^{1/2}(\bD)$ so that
$$ g = e^{ikx} - S_k\left(\Lam_\sig - \Lam_1\right) g. $$
\end{theorem}

As we will see, \eqref{sigma.one} implies that $\Lam_\sig - \Lam_1$ is smoothing even though $\sig$ may be nonsmooth. One can then mimic Nachman's original argument from Fredholm theory to prove the unique solvability. We will show that the ``scattering transform'' generated by \eqref{N.bdry.t.bis} is a natural limit of smooth approximations, and remains related to the Astala-P\"{a}iv\"{a}rinta scattering transform $\tau$ by \eqref{tau-to-t}, even though the Schr\"{o}dinger problem now involves a distribution potential. 

To make this connection, we consider approximation of $\sig \in L^\infty$ by smooth conductivities. In particular, suppose that $\sig$ is a fixed conductivity obeying \eqref{sigma.one} with strictly positive essential infimum, and that $\{\sig_n\}$ is a sequence of smooth conductivities in $\D$ obeying
\begin{itemize}
\item[(i)]		There is a fixed $r_1 \in (0,1)$ so that $\sig_n(x) =1$ for $|x| \geq r_1$ and for all $n$, 
\item[(ii)]	There is a fixed $c>0$ so that $\sig_n(x) \geq c$ for a.e.\ $x \in \D$ and for all $n$,
\item[(iii)]	For a.e.\  $x$, $\sig_n(x)$ is monotone nondecreasing with $\sig_n(x) \to \sig(x)$ 
				as $n \to \infty$. 
\end{itemize}

\begin{theorem}
\label{thm:N.approx}
Suppose that $\{ \sig_n \}$ obeys (i)--(iii), and denote by $\bft_n$ (resp.\ $\bft$) the scattering transform for $\sig_n$ (resp.\ $\sig$) obtained from \eqref{N.bdryint.bis}--\eqref{N.bdry.t.bis}. Then $\bft_n \to \bft$ pointwise. Moreover, $\bft$ is related to the Astala-P\"{a}iv\"{a}rinta scattering transform $\tau$ for $\sig$ by \eqref{tau-to-t}.
\end{theorem}

We will prove Theorem \ref{thm:N.approx} by studying convergence of the operators $\Lam_{\sig_n} - \Lam_1$ to $\Lam_\sig - \Lam_1$ as $n \to \infty$. An important ingredient in the proof will be the fact that the operators $\Lam_{\sig_n} - \Lam_1$ are \emph{uniformly} compact in a sense to be made precise, so that weak convergence (which is relatively easy to prove) can be ``upgraded'' to norm convergence. 

It is then natural to ask, whether, on the other hand, a sequence of cutoff scattering transforms converging to the ``true'' scattering transform of a singular conductivity produces a convergent reconstruction. This question is much harder because the truncated scattering transforms, by analogy with the Fourier transform, generate approximate conductivites that are \emph{not} identically $1$ outside a compact set. This means that the analysis of Astala and P\"{a}iv\"{a}rinta, which exploits the compact support of the Beltrami coefficient $\mu$, must be considerably extended. We will return to this analysis in a subsequent paper.

\subsection*{Acknowledgements} P.P. and G.L. thank the Department of Mathematics, University of Helsinki, for hospitality during part of the time this work was done.

\section{Preliminaries}
\label{sec:prelim}

Here and in what follows, we use the notation $f \lesssim_{\, c} g$ to mean that $f \leq Cg$ where the implied constant $C$ depends on the quantities $c$. 

\subsection{\texorpdfstring{$H^s$ Spaces, Fourier Basis, Harmonic extensions}{Fourier Basis, Harmonic Extensions}} 
An $L^2$ function $f \in L^2(\bD)$ admits a Fourier series expansion
$f(\theta) \sim \sum_n b_n \varphi_n(\theta)$, where
$$ \varphi_n(\theta) = \frac{1}{\sqrt{2\pi}} e^{in\theta}.$$
The equation
\begin{equation}
\label{Pk}
\left( P_j f \right)(\theta) = \sum_{|n| \leq j} b_n \varphi_n(\theta)
\end{equation}
for $j \in \N$ defines a finite-rank projection.
For $s \in \R$, we denote by $H^s(\bD)$ the completion of $C^\infty(\bD)$ in the norm
$$ \norm[H^s(\bD)]{f} = \left( \sum_{n=-\infty}^\infty (1+|n|)^{2s} |b_n|^2 \right)^{1/2}. $$
It is easy to see that the embedding 
\begin{equation}
\label{Hs.compact}
H^s(\bD) \hookrightarrow H^{s'}(\bD)
\end{equation}
is compact provided $s> s'$. 

The harmonic extension of $f \in L^2(\bD)$ to $\D$ is given by 
$$ u(r,\theta) = \sum_{n=-\infty}^\infty r^{|n|} b_n \varphi_n(\theta). $$
It is easy to see that for any $r_1 \in (0,1)$, the estimate
\begin{equation}
\label{harmonic.in}
\norm[L^2(|x| < r_1)]{u} \lesssim_{\, m, r_1} \norm[H^{-m}]{f}
\end{equation}
holds for the harmonic extension.  

\subsection{\texorpdfstring{Faddeev's Green's Function and the operator $S_k$}{Fadeev's Green's Function}}

The Faddeev Green's function is the convolution kernel $G_k(x-y)$ where
\begin{equation}
\label{Gk}
G_k(x)	=	 \frac{e^{ikx}}{(2\pi)^2} \int_{\R^2} \frac{e^{ix \cdot \xi}}{|\xi|^2 + 2k(\xi_1+i\xi_2)} \ d\xi
\end{equation}
where $x\cdot \xi = x_1\xi_1 +x_2\xi_2$. This is the natural Green's function for the elliptic problem \eqref{CGO}. Writing $G_k(x) = e^{ikx} g_k(x)$ we see that $g_k(x)$ differs from the Green's function $G_0(x) = -(2\pi)^{-1} \log |x|$ of the Laplacian by a function which is smooth and harmonic on all of $\R^2$ and, in particular, is regular at $0$ (see, for example, \cite[Section 3.1]{Siltanen99} for further discussion and estimates). 

In what follows we will assume $\Omega \subset \R^2$ is bounded and simply connected with smooth boundary (since our application is to $\Omega = \D$) even though these assertions are known in greater generality.
In the reduction of \eqref{CGO} to the boundary integral equation \eqref{N.bdryint}, the operator $S_k$ is the corresponding single layer
\begin{equation}
\label{Sk}
\left( S_k f \right)(x)  = \int_{\dee \Omega} G_k(x-y)  f(y) \, dy.
\end{equation}
For $p \in (1,\infty)$ and any $f \in L^p(\dee \Omega)$, the function $S_k f$ is smooth and harmonic on $\R^2 \setminus \dee \Omega$.
Moreover, since the convolution kernel $G_k$ is at most logarithmically singular, $S_k f$ restricts to a well-defined a function on $\dee \Omega$. 
When restricted to $\dee \Omega$, 
\begin{equation}
\label{Sk.map}
S_k : H^{s}(\dee \Omega) \to H^{s+1}(\dee \Omega), \quad s \in [-1,0]
\end{equation}
(see  \cite[Lemma 7.1]{Nachman96}), even if $\Omega$ only has Lipschitz boundary.

It follows from the form of $G_k(x)$ and classical potential theory that, if $\nu(x)$ is the unit normal to $\dee \Omega$ at $x \in \dee \Omega$, the identities
\begin{align}
\label{Sk.out}
\lim_{\substack{ ~\\z \to x \\[2pt] z \in \R^2 \setminus \Omega}}
	\left\langle \nu(x), \left(\nabla S_k f\right)(z) \right\rangle 
		&= -\left(\frac{1}{2}I -S_k\right) f(x)\\
\label{Sk.in}
\lim_{\substack{ ~ \\z \to x \\[2pt] z \in\Omega}}
	\left\langle \nu(x), \left(\nabla S_k f\right)(z) \right\rangle 
		&= -\left(\frac{1}{2}I +S_k \right) f(x)	
\end{align}
hold

\subsection{Alessandrini Identity }
We will make extensive use of the following identity \cite{Alessandrini88} which is an easy consequence of Green's theorem. Suppose that $u$ solves \eqref{bvp} and that $v \in H^1(\Omega)$ with boundary trace $g \in H^{1/2}(\dee \Omega)$. Then
\begin{equation}
\label{Alessandrini}
\left\langle g, \Lam_{\sig} f \right\rangle = \int_\Omega \sig(x) (\nabla u)(x) \cdot (\nabla v)(x) \, dx
\end{equation}
where $\langle g,h \rangle$ denotes the dual pairing of $g \in H^{1/2}(\dee \Omega)$ with $h \in H^{-1/2}(\dee \Omega)$. 

\subsection{A Priori Estimates and Uniqueness Theorems}
We'll need the following results from \cite{AIM09} which we state here for the reader's convenience. First, the following \emph{a priori} estimate on solutions of Beltrami's equation to analyze convergence of CGO solutions to the Beltrami equations assuming that the Beltrami coefficients converge pointwise.

\begin{theorem}\cite[Theorem 5.4.2]{AIM09}
\label{thm.aim.caccio}
Let $f\in W^{1,q}_{loc}(\Omega)$, for some $q \in (q_\kappa, p_\kappa) = (1+\kappa, 1+\frac{1}{\kappa})$, satisfy the distortion inequality
\begin{equation*}
\left| \dbar f\right| \leq \kappa \left|\dee f\right|
\end{equation*}
for almost every $x\in \Omega$.  Then $f\in W^{1, p}_{loc}(\Omega)$ for every $p \in (q_\kappa, p_\kappa)$.  In particular, $f$ is continuous, and for every $s \in (q_\kappa, p_\kappa)$, the critical interval, we have the Caccioppoli estimate
\begin{equation}
\label{Caccioppoli.ineq.}
\norm[s]{\eta \nabla f} \leq C_s(k)\norm[s]{f\nabla \eta}
\end{equation}
whenever $\eta$ is a compactly supported Lipschitz function in $\Omega$.
\end{theorem}

The following uniqueness theorem for CGO solutions of the conductivity equation will help establish the unique solvability of the integral equation \eqref{N.bdryint.bis}.

\begin{theorem} \cite[Corollary 18.1.2]{AIM09}
\label{thm:cond.unique}
Suppose that $\sigma, 1/\sigma \in L^{\infty}(\D)$ and that $\sigma(x)\equiv 1$ for $|x|\geq 1$.  Then the equation $\nabla \cdot(\sigma \nabla u))=0$
admits a unique weak solution $u \in W^{1,2}_{loc}(\C)$ such that
\begin{equation}
\lim_{|x| \to \infty}  \left( e^{-ikx}u(x,k)-1 \right) = 0.
\end{equation}
\end{theorem}

\section{Boundary Integral Equation}
\label{sec:bdryint}

In this section we prove Theorem \ref{thm:N.solve}. Our strategy is to show that the integral operator 
$$ T_k \coloneqq S_k (\Lam_\sig-\Lam_1)$$ 
is compact on $H^{1/2}(\bD)$ and then mimic Nachman's argument in \cite[Section 8]{Nachman96} to show that $I + T_k$ is injective. The following simple lemma reduces the compactness statement to interior elliptic estimates plus the property \eqref{harmonic.in} of harmonic extensions.

\begin{lemma}
\label{lemma:D.to.N.id}
For any $f$ and $g$ belonging to $H^{1/2}(\bD)$, the identity
\begin{equation}
\label{D.to.N.id}
\left\langle g, (\Lambda_\sigma - \Lambda_1) f \right\rangle = \int_\D (\sigma -1 ) \nabla v \cdot \nabla u \, dx
\end{equation}
holds, where $u$ solves \eqref{bvp} and $v$ is the harmonic extension of $g$ to $\D$ and $\langle g,h\rangle$ denotes the dual pairing of $g \in H^{1/2}(\bD)$ with $h \in H^{-1/2}(\bD)$.
\end{lemma}

\begin{proof}
Let $	w$ be the harmonic extension of $f$ to $\D$. It follows from Alessandrini's identity \eqref{Alessandrini} that
\begin{align*}
\left\langle g, (\Lambda_\sigma - \Lambda_1) f \right\rangle 
	&=	\int_\D \sigma \nabla v \cdot \nabla u \, dx - \int_\D \nabla v \cdot \nabla w \, dx\\
	&=	\int_\D (\sigma-1) \nabla v \cdot \nabla u \, dx + \int_\D \nabla v \cdot \nabla (u -w ) \, dx
\end{align*}
The second term vanishes since $v$ is harmonic and $ \left. (u-w) \right|_{\bD} = 0$. 
\end{proof}

Next, we note the following interior elliptic estimate.

\begin{lemma}
\label{elliptic.est}
Suppose that $\sigma$ satisfies \eqref{sigma.one}, let $f \in H^{1/2}(\bD)$, and let $u$ denote the unique solution of \eqref{bvp} for the given $f$. For any $m>0$, the estimate
\begin{equation}
\norm[L^2(|x|<r_1)]{\nabla u}\lesssim \norm[H^{-m}(\bD)]{f}
\end{equation}
holds,
where the implied constant depends only on 
$m$, $r_1$, $\essinf \sigma$, and $\esssup \sigma$.
\end{lemma}
\begin{proof}
As before, let $w$ be the harmonic extension of $f$ into $\D$. 
Let $r_1$ be the radius defined in \eqref{sigma.one}, and let $0<r_1<r_2<1$. Choose $\chi \in C^\infty(\overline{\D})$ so that
\begin{equation}
\label{chi.def}
\chi(x) = \begin{cases}
0, & 0\leq |x| \leq r_1\\
1, & r_2\leq |x|\leq 1
\end{cases}
\end{equation}
Let $h(x) = \chi(x)w(x)$.  Note that $h$ has support where $\sigma(x) = 1$.  We compute
\begin{align*}
\nabla \cdot(\sigma \nabla (u-h)) &= \nabla \cdot(\sigma \nabla u) - \nabla\cdot(\sigma\nabla(h))\\
&= -(\Delta \chi)w - 2\nabla \chi \cdot \nabla w
\end{align*}
By construction, we know $\bdrytrace{(u-h)} = 0$.  

The unique solution $v \in H^1_0(\D)$ of
$$ \nabla \cdot \left( \sigma \nabla v \right) = g $$
obeys the bound
$$ \norm[L^2(\D)]{\nabla v} \lesssim \norm[L^2(\D)]{g} $$
where the implied constants depend only on $\essinf \sigma$ and $\esssup \sigma$. Hence
\begin{equation*}
\norm[L^2(|x|<r_1)]{\nabla u} = \norm[L^2(|x|<r_1)] {\nabla(u-w)}\lesssim \norm[L^2(\D)]{-(\Delta \chi)h - 2\nabla \chi \cdot \nabla h}
\end{equation*}
We obtain the desired estimate using \eqref{harmonic.in}.
\end{proof}

Next, we prove an operator bound on $(\Lam_\sig-\Lam_1)$ with a uniformity that will be useful later.

\begin{lemma}
\label{lemma:DN.map.smoothing}
 Let $\sigma \in L^\infty(\D)$ with $\sigma(x) \geq c >0$ a.e. for some constant $c$. Suppose, moreover, that $\sigma$ obeys \eqref{sigma.one}. Then for any $m>0$, the operator $(\Lam_\sigma - \Lam_1)$ is bounded from $H^{-m}(\bD)$ to $H^m(\bD)$ with constants depending only on $r_1$, $m$, $\essinf \sigma$, and $\esssup \sigma$. 
\end{lemma}
\begin{proof}
We will begin with $f, g\in H^{1/2}(\bD)$ and show that the pairing
\begin{equation*}
|\left\langle g, (\Lambda_\sigma - \Lambda_1) f \right\rangle|
\end{equation*}
can be bounded in terms of $\norm[H^{-m}]{f}$ and $\norm[H^{-m}]{g}$.  Then a density argument will establish the lemma.  

Let $v$ be a harmonic extension of $g$ into $\D$. Then by Lemma \ref{lemma:D.to.N.id} we obtain
\begin{align*}
\left| \left( g, (\Lambda_\sigma - \Lambda_1) f \right)\right| 
		&= \left|\int_\D (\sigma -1 ) \nabla v \cdot \nabla u \, dx \right|\\[5pt]
		&\lesssim_{\, \sigma}  \norm[L^2(|x|< r_1)]{\nabla u} \norm[L^2(|x| < r_1)]{\nabla v}\\[5pt]
		&\lesssim_{\, \sigma, r_1, m} \norm[H^{-m}]{f}\norm[H^{-m}]{g}
\end{align*}
where we used Lemma \ref{elliptic.est} to estimate $\norm[L^2(|x| < r_1)]{\nabla u}$ and we used \eqref{harmonic.in} again to estimate $\norm[L^2(|x|< r_1)]{\nabla v}$.  The implied constants depend only on $\essinf \sig$ and $\esssup \sig$.
\end{proof}

It now follows from Lemma \ref{lemma:DN.map.smoothing} and the compact embedding \eqref{Hs.compact} that $T_k$ is compact as an operator from $H^{1/2}(\bD)$ to $H^{-1/2}(\bD)$. Thus, to show that \eqref{N.bdryint.bis} is uniquely solvable, it suffices by Fredholm theory to show that the only vector $g \in H^{1/2}(\bD)$ with $g=-T_k g$ is the zero vector. We will show that any such $g$ generates a global solution to the problem 
\begin{equation}
\label{cgo.calderon}
\begin{aligned}
\nabla \cdot \left( \sigma \nabla u  \right) 	&=	0,	\\
\lim_{|x| \to \infty}  e^{-ikx}u(x,k) &= 0.
\end{aligned}
\end{equation}
We will then appeal to Theorem \ref{thm:cond.unique} to conclude that $g=0$.

\begin{proof}[Proof of Theorem \ref{thm:N.solve}]
We follow the proof of Theorem 5 in \cite[Section 7]{Nachman96}. Fix $k \in \C$, suppose that $g \in H^{1/2}(\bD)$ satisfies $T_k g = -g$, let $h = (\Lam_\sig - \Lam_1)g$ and let $v = S_k h$ on $\R^2 \setminus \bD$. The function $v$ is harmonic on $\R^2 \setminus \bD$ and continuous across $\bD$. Thus, if 
$v_+$ and $v_-$ are the respective boundary values of $v$ from $\R^2 \setminus \bD$ and from $\bD$, $v_+ = v_- = g$.   It follows from \eqref{Sk.out}--\eqref{Sk.in} and the fact that $g = -T_k g$ that 
\begin{equation}
\label{jump.v}	
\frac{\dee v_+}{\dee \nu} - \frac{\dee \nu_-}{\dee \nu} = h = \Lam_\sig g - \Lam_1 g.
\end{equation}
Since $\dee \nu_-/\dee \nu = \Lam_1 g$, we conclude that $\dee \nu_+ /\dee \nu = \Lam_\sig g$. 
Now define
$$ u(x) = 
\begin{cases}
v(x), 	&	x \in \R^2 \setminus \Omega\\
u_i(x),	&	x \in \Omega
\end{cases}
$$
where $u_i$ is the unique solution to the problem 
$$ 
\nabla \cdot \left( \nabla u_i \right) = 0, \quad \left. u \right|_{\bD} = g.
$$
In this case $u_+ = u_-$ and $\dee u_+ /\dee \nu = \dee u_-/\dee \nu$, so $u$ extends to a solution of \eqref{cgo.calderon} as claimed. It now follows from Theorem \ref{thm:cond.unique}
that $u = 0$. Since $g$ is the boundary trace of $u$, we conclude that $g=0$. 
\end{proof}

\section{Convergence of Scattering Transforms}
\label{sec:scattering}

In this section we prove Theorem \ref{thm:N.approx} in two steps. First, we show that the Dirichlet-to-Neumann operators $\Lam_{\sig_n}$ associated to the sequence $\{ \sig_n \}$ converge in norm to $\Lam_\sig$. We then use this fact to conclude that the corresponding scattering transforms converge. The second step uses Astala-P\"{a}iv\"{a}rinta's scattering transform to identify the limit.

We begin with a simple result on weak convergence that exploits Alessandrini's identity and convergence of positive quadratic forms.

\begin{lemma}
\label{lemma:D-to-N.weak}
Suppose that $\{ \sigma_n \}$ is a sequence of positive $L^\infty(\D)$ obeying conditions (i)--(iii) of Theorem \ref{thm:N.approx}. 
Then $\Lambda_{\sigma_n} \to \Lambda_\sigma$ in the weak operator topology on $\calL(H^{1/2}(\bD), H^{-1/2}(\bD))$.
\end{lemma}

\begin{proof}
For any $\sigma$, it follows from \eqref{Alessandrini} that $\Lam_\sig$ defines a positive 
quadratic form
$$ \left\langle f , \Lam_\sig f \right\rangle  =\int_{\D} \sig \left| \nabla u \right|^2 \, dx $$
on $H^{1/2}(\bD)$. Moreover, by monotone convergence, the quadratic forms $\Lam_\sig - \Lam_{\sig_n}$ are nonnegative for all $n$. If we can show that 
\begin{equation}
\label{D-to-N.weak.diag}
\lim_{n \to \infty} \left\langle f, \left( \Lam_\sig - \Lam_{\sig_n} \right) f  \right\rangle = 0
\end{equation}
it will then follow by polarization that $\Lam_{\sig_n} \to \Lam_\sig$ in the weak operator topology. But 
\begin{equation}
\label{DN.weak.diag}
\left\langle f, \left( \Lam_\sig - \Lam_{\sig_n} \right) f \right\rangle
	= \int_\D (\sig-\sig_n) \left| \nabla u \right|^2 \, dx +
		\int_\D \sig_n \left( \left| \nabla u \right|^2 - \left| \nabla u_n \right|^2\right)  \ dx.
\end{equation}
The first right-hand term in \eqref{DN.weak.diag} goes to zero by monotone convergence. Since the $\{ \sig_n \}$ are uniformly bounded, it suffices to show that $\nabla u_n  \to \nabla u$ in $L^2$. A straightforward computation shows that
$$ 0 = \nabla \cdot \left( \sig_n \nabla (u_n - u )\right) + \nabla \cdot \left( (\sig_n -\sig) \nabla u \right). $$
Multiplying through by $v_n =u_n - u$ and integrating over $\D$, we obtain 
\begin{equation}
\label{ipid}
\int_\D \sig_n \left| \nabla v_n \right|^2 \ dx = - \int_\D  (\sig_n - \sig) \nabla v_n \cdot \nabla u \, dx. 
\end{equation}
Since $\sig_n$ is bounded below by a fixed positive constant $c$ independent of $n$,
we can use the Cauchy-Schwarz inequality to conclude that
$$
\frac{c}{2} \int_\D \left| \nabla v_n \right|^2 \, dx 
		\leq \frac{1}{2c} \int_D (\sig-\sig_n) \left| \nabla u \right|^2 \ dx
$$
and conclude that $\nabla u_n \to \nabla u$ in $L^2$ by monotone convergence.
\end{proof}

From Lemma \ref{lemma:DN.map.smoothing} we obtain the following uniform approximation property for the operators 
\begin{equation}
\label{An}
A_n \coloneqq \Lam_{\sig_n} - \Lam_1.
\end{equation}

\begin{lemma}
Suppose that $\{ \sig_n \}$ is a sequence of conductivities obeying hypotheses (i)--(iii) of
Theorem \ref{thm:N.approx}, and let $A_n$ be defined as in \eqref{An}.
Given any $\eps>0$ there is a $k \in \N$ independent of $n$ so that
$$
\norm[H^{1/2} \to H^{-1/2}]{(I-P_k) A_n} < \eps, \quad \
\norm[H^{1/2} \to H^{-1/2}]{A_n(I - P_k)} < \eps, 
$$
\end{lemma}

\begin{proof}
From Lemma \ref{lemma:DN.map.smoothing} we have the uniform operator bound
$ \norm[H^m \to H^{-m}]{A_n} \lesssim_{\, m} 1$ since the $\sig_n$ have uniformly bounded
essential infima and suprema and all obey \eqref{sigma.one}. If $A_n'$ denotes the Banach space adjoint of $A_n$, we have the same bound on $A_n'$ by duality. The second bound is equivalent to the bound 
$$\norm[H^{1/2} \to H^{-1/2}]{(I-P_k) A_n'} < \eps$$ 
by duality, so we'll only prove the first bound. We write
\begin{align*}
\norm[H^{1/2} \to H^{-1/2}]{(I-P_k) A_n}
	&\leq	\norm[H^m \to H^{1/2}]{(I-P_k)} \norm[H^{-m} \to H^m]{A_n}	\\
	&\lesssim_{\, m} k^{1/2-m}
\end{align*}
with constants uniform in $n$. 
\end{proof}

Now let $A = \Lam_\sig - \Lam_1$ where $\sig_n \to \sig$. 

\begin{proposition}
\label{prop:D-to-N-strong}
Suppose that $\{ \sig_n \}$satisfies hypotheses (i)--(iii) of Theorem \ref{thm:N.approx}. 
Then $A_n  \to A$ in the norm topology on the bounded operators from $H^{1/2}$ to $H^{-1/2}$.
\end{proposition}

\begin{proof}
Write
\begin{equation}
\label{trichotomy}
A_n - A = P_k (A_n - A) P_k + (I-P_k) (A_n  - A) + (A_n - A)(I-P_k). 
\end{equation}
Since $A$ is a fixed compact operator, we can choose $N \in \N$ so  $\norm[H^{1/2} \to H^{-1/2}]{(I-P_k) A}$ and $\norm[H^{1/2} \to H^{-1/2}]{A(I-P_k)}$ are small for any $k \geq N$. Combining this observation with Proposition \ref{prop:D-to-N-strong}, we can choose $k \in \N$, uniformly in $n$, so that the first and and third right-hand terms of \eqref{trichotomy} are small uniformly in $n$. The middle term vanishes for any fixed $k$ and $n \to \infty$ by Lemma \ref{lemma:D-to-N.weak}.
\end{proof}

As an easy consequence:
\begin{proposition}
Fix $k \in \C$. 
Suppose that $\{ \sig_n \}$ is a sequence obeying hypotheses (i)--(iii) of Theorem \ref{thm:N.approx}, and denote by $g_n(\dotarg,k)$ and $g(\dotarg,k)$ the respective  solutions of \eqref{N.bdryint.bis} corresponding to $\sig_n$ and $\sig$. Then, for each fixed $k$,  $g_n \to g$ in $H^{1/2}(\bD)$. Moreover, the scattering transforms $\bft_n$ of $\sig_n$ converge pointwise to $\bf$ given by \eqref{N.bdry.t.bis}.
\end{proposition}

\begin{proof}
By a slight abuse of notation, denote by $T_n$ the operator $S_k \left(\Lam_{\sig_n} - \Lam_1\right)$ and by $T$ the operator $S_k \left(\Lam_{\sig} - \Lam_1\right)$. It follows from \eqref{Sk.map} that $T_n \to T$ in $\calL(H^{1/2},H^{1/2})$. Since
$$ g_n = (I-T_n)^{-1} \left(\left.e^{ikx}\right|_{\bD}\right), \quad g = (I-T)^{-1} \left(\left.e^{ikx}\right|_{\bD} \right), $$
it follows from the second resolvent identity that $g_n \to g$ in $H^{1/2}(\bD)$. Convergence of $\bft_n$ to $\bft$ follows from the norm convergence of $g_n$ to $g$ and of $\Lam_{\sig_n} - \Lam_1$ to $\Lam_\sig -\Lam_1$. 
\end{proof}

In the remainder of this section, we will identify what $\bft$ actually is. In order to do so we need to prove a convergence theorem for the Astala-P\"{a}iv\"{a}rinta scattering transforms $\tau_n$ of the Beltrami coefficients $\mu_n = (1-\sig_n)/(1+\sig_n)$ to the transform $\tau$ of $\sig$ that is of some interest in itself.

\begin{proposition}
\label{prop:Beltrami}
Suppose that $\{\mu_n\}$ is a sequence of Beltrami coefficients with $0 \leq \mu_n(x) \leq \kappa$ for a.e.\ $x$ and $0 \leq \kappa < 1$. Suppose further that  $\mu_n(x) \to \mu(x)$ where $\mu \in L^\infty(\D)$ has the same properties. Finally, fix $k \in \C$ and let 
$M_{\pm \mu_n}(x,k)$ be the normalized CGO solution for the Beltrami equation \eqref{Beltrami} with Beltrami coefficients $\pm \mu_n$, and let $M_{\pm u}$ be the normalized CGO solution for $\pm \mu$.  Then, for a single choice of sign,
$M_{\pm \mu_n} - 1 \to M_{\pm \mu} - 1$ weakly in $W^{1,p}(\R^2)$ for any $p \in (2,1+\kappa^{-1})$. 
\end{proposition}

We will prove Proposition \ref{prop:Beltrami} in several steps. First we show how to conclude the proof of Theorem \ref{thm:N.approx} given its result.

\begin{proof}[Proof of Theorem \ref{thm:N.approx}, given Proposition	\ref{prop:Beltrami}]
Proposition \ref{prop:Beltrami} and \eqref{tau} show that $\tau_{\mu_n} \to \tau$ pointwise as $n \to \infty$ since the integral in \eqref{tau} may be regarded as integrating the derivatives of $M_{\pm \mu_n}$ (which, by \eqref{Beltrami}, are supported in the unit disc) against a smooth, compactly supported function which is identically $1$ in a neighborhood of $\D$. Since $\tau_n$ converges pointwise to $\tau$ and $\bft_n(k) = -4\pi i \overline{k}\tau_{\mu_n}(k)$, we conclude that $\bft(k) = -4\pi i \overline{k}\tau(k)$.
\end{proof}

To establish the weak convergence, we first need a uniform bound on $M_{\pm \mu_n} -1$ in 
$W^{1,p}(\R^2)$.

\begin{lemma}
\label{lemma:Mn-bounded}
Suppose that $\{ \mu_n \}$ is a sequence of Beltrami coefficients obeying the hypothesis of Proposition \ref{prop:Beltrami}, and let $M_n = M_{\mu_n}$.  Then there exists a constant $C$ such that
\begin{equation}
\sup_{n}\norm[W^{1,p}(\R^2)]{M_n-1} < C.
\end{equation}
\end{lemma}

\begin{proof} Let $c_n = \norm[W^{1,p}(\R^2)]{M_n - 1}$. If $c_n \to +\infty$ as $n \to \infty$, 
set $v_n = c_n^{-1} (M_n - 1)$. Since $\{ v_n \}$ is bounded in $W^{1,p}$, by passing to a subsequence we can assume that $\{ v_n \}$ has a weak limit, $v$. Note that $\norm[W^{1,p}(\R^2)]{v_n} = 1$. 

We first claim that, if such a limit exists, it is nonzero. Suppose, on the other hand, that $v_n \to 0$ weakly in $W^{1,p}(\R^2)$. It follows from the Rellich–Kondrachov Theorem that $v_n \to 0$ in $L^p_{\loc}(\R^2)$ .A short computation shows that
\begin{equation}
\label{dbar.vn}
\dbar v_n = \frac{\mu_n}{c_n} \dbar e_k + \mu_n \overline{\dee(e_k v_n)} 
\end{equation}
and, since $v_n \in W^{1,p}(\R^2)$ we may invert the $\dbar$ operator using the Cauchy transform and use standard estimates on the Cauchy transform (see, for example, \cite[Theorem 4.3.8]{AIM09}) to conclude that
\begin{align}
\label{vn.Cauchy.est}
\norm[L^p(\R^2)]{v_n} 
	&\lesssim_{\, p}\norm[L^{2p/(2+p)}(\R^2)]{\frac{\mu_n}{c_n} \, \dbar e_k }
				+ \norm[L^{2p/(2+p)}(\R^2)]{\mu_n \overline{(\dee e_k) v_n}} \\[5pt]
\nonumber
	&\quad
				+ \norm[L^{2p/(2+p)}(\R^2)]{\mu_n \overline{e_k \dee v_n} }.
\end{align}
The first right-hand term in \eqref{vn.Cauchy.est} clearly goes to zero as $n \to \infty$ since $c_n \to \infty$.  The function in the second term is supported in $\D$ owing to the factor $\mu_n$ and therefore also converges to zero since $v_n \to 0$ in $L^p_{\loc}(\R^2)$. The function in the third term is again supported in $\D$ and, using a version of the Caccioppoli inequality adapted to the $v_n$'s (see Lemma \ref{lemma:v_n.est} below), we have $\norm[L^p(\D)]{\dee v_n} \lesssim \norm[L^p(2\D)]{v_n} + \bigO{c_n^{-1}}$, which shows that the third term also goes to zero as $n \to \infty$. Thus, $v_n \to 0$ in $L^p(\R^2)$. Applying Lemma \ref{lemma:v_n.est} to 
the compactly supported function $\dbar v_n$ shows that, also $\norm[L^p]{\dbar v_n} \to 0$ as $n \to \infty$, contradicting the fact that $\norm[W^{1,p}(\R^2)]{v_n} = 1$ for all $n$. From this contradiction we conclude that $\{ v_n \}$ has a nonzero limit, again assuming that $c_n \to \infty$. 

Next, we show that the limit function $v$ is a weak solution of the equation 
$\dbar v = \mu \overline{\dee (e_k v)}$. For $\varphi \in C_0^\infty(\R^2)$ we compute from \eqref{dbar.vn} 
\begin{align*}
(\varphi, v_n) = c_n^{-1}( \varphi, \mu_n \dbar e_k) + (\varphi, \mu_n \overline{\dee (e_k v_n)}).
\end{align*}
where $(f,g) = \dint fg$. 
It is easy to see that the first right-hand term vanishes as $n \to \infty$. In the second term,
\begin{align*}
(\varphi, \mu_n \overline{\dee (e_k v_n)}) 
	&=  	(\dbar(e_{-k})\mu \varphi, \bar{v}_n) + 
			(e_k \mu \varphi, \overline{\dee v_n}) + 
			(\varphi(\mu_n -\mu), \overline{\dee(e_kv_n)})	\\
	&\to	(\dbar(e_{-k}) \mu \varphi, \vbar) + (e_k \mu \varphi, \overline{\dee v})
\end{align*}
since $\norm[W^{1,p}]{v_n} = 1$ and $v_n \to v$ in $L^p_{\loc}$. It follows that $v$ is a weak solution of $\dbar v = \mu \overline{\dee(e_k v)}$ with $\norm[W^{1,p}]{v} \leq 1$. 

Thus, assuming that $\norm[W^{1,p}(\R^2)]{M_n - 1}$ is not bounded, we have constructed a nonzero solution $v \in W^{1,p}(\R^2)$ of the equation $\dbar v = \mu \overline{\dee(e_k v)}$. However, this violates the uniqueness of the normalized CGO solution for Beltrami coefficient $\mu$ proved in \cite[Theorem 4.2]{AP06a}, a contradiction. We conclude that $\norm[W^{1,p}(\R^2)]{M_n - 1}$ is bounded uniformly in $n$.
\end{proof}

To complete the proof of Lemma \ref{lemma:Mn-bounded}, we need to establish the a priori bounds on the sequence $v_n$ constructed in its proof. To do so, we will need the \emph{a priori} estimate for solutions of the Beltrami equation from Theorem \ref{thm.aim.caccio}.

\begin{lemma}
\label{lemma:v_n.est}
Suppose that $v_n$ is a sequence of functions as constructed in the proof of Lemma \ref{lemma:Mn-bounded}. Then, the estimate
$$ \norm[L^p(\D)]{ \dee v_n} \lesssim c_n^{-1} + \norm[L^p(2\D)]{v} $$
where the implied constants are independent of $n$.
\end{lemma}

\begin{proof}
By construction, the function $f = e^{ikz}(c_nv_n+1)$ satisfies the Beltrami equation
\begin{equation*}
\dbar f = \mu_n \overline{\dee f}
\end{equation*}
and hence satisfies the distortion inequality.  Thus by Theorem \ref{thm.aim.caccio}, for a compactly supported smooth function $\eta$ we can write
\begin{equation}
\label{est.1}
\norm[p]{\eta \dee f} \leq C_{p, \kappa}\norm[p]{f\nabla \eta}
\end{equation}
Note that by the triangle inequality, we have
\begin{align}
\norm[p]{\eta \dee f} = \norm[p]{\eta\dee(e^{ikz}(c_nv_n+1))}
	&	\geq c_n\norm[p]{\eta e^{ikz}\dee v_n} \\
\nonumber
	&\quad - |k|\norm[p]{\eta e^{ikz}(c_nv_n+1)}
\end{align}
which enables us to write
\begin{align}
c_n\norm[p]{\eta e^{ikz}\dee v_n}
	&	\leq \norm[p]{\eta \dee(e^{ikz}(c_nv_n+1))}	\\
\nonumber
	&\quad  + |k|\norm[p]{\eta e^{ikz}} + c_n|k|\norm[p]{\eta e^{ikz} v_n}
\end{align}
Next, we apply \eqref{est.1} to obtain
\begin{align}
c_n\norm[p]{\eta e^{ikz}\dee v_n} 
	&	\lesssim  c_n\norm[p]{e^{ikz}(\nabla \eta) v_n} + \norm[p]{e^{ikz}\nabla \eta} \\
\nonumber
	&\quad + |k|\norm[p]{\eta e^{ikz}} + c_n|k|\norm[p]{\eta e^{ikz} v_n}.
\end{align}
Thus we conclude
\begin{align}
\norm[p]{\eta e^{ikz}\dee v_n}&  \lesssim  \norm[p]{e^{ikz}(\nabla \eta) v_n} + |k|\norm[p]{\eta e^{ikz}v_n} \\
\nonumber
&\quad + \frac{1}{c_n}\left(\norm[p]{e^{ikz}\nabla \eta} 
	+ |k|\norm[p]{\eta e^{ikz}}\right)
\end{align}
To obtain the desired estimate, we choose $\eta$ supported on the disk of radius 2 so that $\eta = e^{-ikz}$ in $\D$.
\end{proof}

We can now give the proof of Proposition \ref{prop:Beltrami} and thereby complete the proof of Theorem \ref{thm:N.approx}.

\begin{proof}[Proof of Proposition \ref{prop:Beltrami}]
By Lemma \ref{lemma:Mn-bounded}, the sequences $\{ M_{\pm \mu_n} - 1\}$ for either choice of sign are bounded in $W^{1,p}(\R^2)$. We will take a single choice of sign, the $+$ sign, and write $M_n$ for $M_{\mu_n}$ and $M$ for $M_\mu$ from now on. The sequence $\{ M_n - 1\}$ has a weak limit point in $W^{1,p}(\R^2)$ which we denote by $M^\sharp-1$. By the Rellich-Kondrachov theorem, $M_n-1$ converges in $L^p_{\loc}(\R^2)$ to  $M^\sharp-1$.  We wish to show that 
\begin{equation}
\label{Msharp.pde}
\begin{aligned}
\dbar M^\sharp	&=	\mu \overline{\dee \left( e_k M^\sharp \right)},\\
M^\sharp -1		&\in 	W^{1,p}(\R^2)
\end{aligned}
\end{equation}
since we can then conclude that $M^\sharp -1$ is nonzero (as the PDE does not admit the solution $M^\sharp =1$) and that $M^\sharp = M$ since the PDE is uniquely solvable for $M^\sharp-1 \in W^{1,p}(\R^2)$. 

From $\dbar M_n = \mu_n \overline{\dee \left( e_k M_n \right)}$ we conclude that for any $\varphi \in C_0^\infty(\R^2)$, 
\begin{align*}
-\left( \dbar \varphi, M_n \right) 
	&= \left(\varphi, \mu_n \overline{ \dee \left( e_k M_n \right)} \right)\\
	&= \left( \varphi, \mu  \overline{ \dee \left( e_k M_n \right)} \right) +
		\left((\mu_n-\mu),\overline{ \dee \left( e_k M_n \right)} \right)
\end{align*}
The second right-hand term vanishes as $n \to \infty$ by dominated convergence since $\mu_n - \mu$ is supported in $\D$ while $\{ \dee (e_k M_n) \}$ is uniformly bounded in $L^p(\R^2)$. Weak convergence of derivatives allows us to conclude that \eqref{Msharp.pde} holds. 
\end{proof}

\bibliographystyle{amsplain}
\bibliography{LPS}{}

\end{document}